\documentclass[a4paper, 10pt]{article}
  \usepackage{amssymb}
\usepackage{amsmath, amsthm, enumitem}
\usepackage{verbatim}
\PassOptionsToPackage{hyphens}{url}\usepackage{hyperref}
\usepackage{graphicx}
\usepackage{setspace}
\usepackage{mathrsfs, todonotes}
\usepackage[all]{xy}

\newcommand{\Spectra}{\mathrm{Sp}}
\newcommand{\Sp}{\Spectra}

\newcommand{\ra}{\rightarrow}

\newcommand{\Fun}{\mathrm{Fun}}
\newcommand{\Map}{\text{Map}}

\newcommand{\Wedge}{\vee}
\newcommand{\BigWedge}{\bigvee}

\newcommand{\x}{\times}

\DeclareMathOperator{\colim}{colim}

\DeclareMathOperator{\sk}{\mathrm{sk}}

\newcommand{\xscr}{\mathscr{X}}

\newcommand{\cscr}{\mathscr{C}}
\newcommand{\dscr}{\mathscr{D}}

\newcommand{\mscr}{\mathscr{M}}

\newcommand{\ltri}{\left<}
\newcommand{\rtri}{\right>}

\DeclareMathOperator{\Alg}{Alg}
\theoremstyle{plain}
\newtheorem{thm}{Theorem}[section]
\newtheorem{cor}[thm]{Corollary}
\newtheorem{prop}[thm]{Proposition}
\newtheorem{lem}[thm]{Lemma}

\newtheorem{defn}{Definition}

\begin{document}
\title{Monadicity of the Bousfield-Kuhn functor} 
\author{Rosona Eldred, Gijs Heuts, Akhil Mathew, and Lennart Meier}
\maketitle

\begin{abstract}
Let $\mscr_n^f$ be the localization of the $\infty$-category of spaces at the $v_n$-periodic equivalences, the case $n=0$ being rational homotopy theory. We prove that $\mscr_n^f$ is for $n\geq 1$ equivalent to algebras over a certain monad on the $\infty$-category of $T(n)$-local spectra. This monad is built from the Bousfield--Kuhn functor.  
\end{abstract}


\section{Introduction}
We fix a prime $p$ and work $p$-locally throughout this introduction. A map $f\colon X \to Y$ between simply-connected $p$-local spaces is a rational homotopy equivalence if $\pi_*(X)[\tfrac1p] \to \pi_*(Y)[\tfrac1p]$ is an isomorphism. Rational homotopy theory concerns itself with the localization of the $\infty$-category of (simply connected) spaces at the rational homotopy equivalences. Quillen provided both a coalgebraic model of rational homotopy theory (via cocommutative coalgebras in rational chain complexes) and an algebraic model (via Lie algebras in rational chain complexes). Under finite type assumptions, one can also dualize the coalgebra model to a cochain model in commutative differential graded algebras, an approach pursued by Sullivan.

From the point of view of chromatic homotopy theory, rational homotopy is only the first step in a sequence of `telescopic' localizations. Such localizations have been studied by Mahowald \cite{MahowaldJ}, Thompson \cite{Thompson}, Davis \cite{Davis} and Bousfield \cite{Bou01}, to name just a few. Given a finite type $n$ complex $V$ with $v_n$ self-map $v\colon \Sigma^d V \to V$, one defines the $v$-periodic homotopy groups of a pointed space $X$ with coefficients in $V$, denoted
$$
v^{-1}\pi_*(X;V),
$$
by inverting the action of $v$ on the homotopy groups of the space of pointed maps $\mathrm{Map}_*(V,X)$. Maps inducing isomorphisms in $v$-periodic homotopy groups are called \emph{$v_n$-equivalences}; the asymptotic uniqueness of $v_n$ self-maps \cite{hopkinssmith} implies that this notion depends only on $n$ and not on $V$ or $v$. Localizing the $\infty$-category of $p$-local pointed spaces at the $v_n$-equivalences results in an $\infty$-category for which we write $\mscr_n^f$. It was first studied by Bousfield \cite{Bou01} (although with different notation). It is an unstable analogue of the category of $T(n)$-local spectra $\Sp_{T(n)}$, where $T(n)$ denotes the mapping telescope of a $v_n$ self-map on a finite type $n$ \emph{spectrum} (rather than space).

Bousfield and Kuhn \cite{KuhnTelescopic} constructed a functor
\[\Phi\colon \mscr_n^f \to \Sp_{T(n)}\]
with the property that $\pi_*(\Phi(X)^V) \cong v^{-1}\pi_*(X; V)$. A map of pointed spaces is then a $v_n$-periodic equivalence if and only if it is sent to an equivalence of spectra by $\Phi$. Behrens and Rezk \cite{BehrensRezk} relate $\Phi(X)$ to the topological Andr\'{e}-Quillen cohomology of the nonunital $E_\infty$-algebra $S_{K(n)}^X$. The assignment $X \mapsto S_{K(n)}^{X}$ can be thought of as a `cochain model' which is analogous to the Sullivan model in rational homotopy theory. 


We will study the Bousfield-Kuhn functor directly and show that it provides an algebraic model for $\mscr_n^f$, which is closely related to Quillen's Lie model of rational spaces. By \cite{heuts-periodic} and \cite{Bou01}, the Bousfield--Kuhn functor is the right adjoint of an adjunction

\[
\xymatrix{
\mscr_n^f \ar@<-3pt>[r]_{\Phi}&\ar@<-3pt>[l]_{\Theta} \Sp_{T(n)}
}
\]
between $\infty$-categories. This adjunction in particular gives a monad $\Phi\Theta$ on the $\infty$-category $\Sp_{T(n)}$, to which one associates a category of algebras $\Alg_{\Phi\Theta} (\Sp_{T(n)})$. Our main result is the following:

\begin{thm}
The Bousfield-Kuhn functor $\Phi$ exhibits $\mscr_n^f$ as monadic over $\Sp_{T(n)}$, that is,  
\[
\mscr_n^f \simeq  \Alg_{\Phi\Theta} (\Sp_{T(n)})
\]
as $\infty$-categories.
\end{thm}

The main step in the proof is to show that $\Phi$ commutes with geometric realizations. In forthcoming work, the second author will identify the monad $\Phi\Theta$ with the free Lie algebra monad (in a sense appropriate to the present context), which proves that $\mscr_n^f$ is equivalent to the $\infty$-category of Lie algebras in $T(n)$-local spectra. This is the parallel between our approach here and Quillen's Lie algebra model for rational homotopy theory. 

Our plan for this paper is as follows. In Sections 2 and 3 we give some background on the $\infty$-category $\mscr_n^f$ and the $\infty$-categorical Barr-Beck monadicity theorem. In Section 4 we prove our main theorem. Throughout this paper, we will work in the language of $\infty$-categories. In particular, $\colim$ means an $\infty$-categorical colimit, $|-|$ means an $\infty$-categorical colimit over a simplicial diagram etc., although we will sometimes add the word `homotopy' for emphasis. We call a space \emph{finite} if it is weakly equivalent to a finite CW-complex.

\subsection*{Acknowledgments} 

We would like to thank Mark Behrens for helpful discussions. 
This work was begun through a Junior Trimester
Program at the Hausdorff Institute of Mathematics, and we thank the HIM for
its hospitality. The third author was supported by the NSF Graduate Fellowship
under grant DGE-114415, and was a Clay Research Fellow when this work was
finished. The fourth author was supported by DFG SPP 1786.

\section{The category $\mscr_n^f$}\label{sec:mnf}

In this section we summarize the basics of unstable telescopic homotopy theory, following Bousfield \cite{Bou01} and the forthcoming \cite{heuts-periodic}. We will follow the notation of the latter and use, in particular, the notation $\mscr_n^f$ for the $\infty$-categorical analogue of Bousfield's $\mathcal{UN}^f_n$. Everything we do is implicitly localized at a fixed prime $p$. 

As in the stable case, one may approximate a space $X$ at a prime $p$ by a tower $\{L_n^f X\}_{n \geq 0}$ of left Bousfield localizations away from finite $p$-local type $n+1$ spaces $V_{n+1}$. To be precise, a finite space $E$ is \textit{type $n$} when the $i$th Morava K-theory of it is trivial for $i<n$, but non-trivial for $i=n$. We choose for every $n \geq  0$ a finite $p$-local type $n+1$ suspension space $V_{n+1}$ and write $L_n^f$ for the left Bousfield localization with respect to the map $V_{n+1} \rightarrow *$. The localization $L_n^f$ depends only on the connectivity $d_{n+1}$ of the $V_{n+1}$. Recall that a space $Y$ is \textit{$m$-connective} when $\pi_k Y \simeq \ast$ for all $k <m$. We denote by $L_n^f \mathcal{S}_\ast \ltri d_{n+1} \rtri$ the $\infty$-category of $L_n^f$-local pointed spaces that are $d_{n+1}$-connected (i.e.\ $(d_{n+1}+1)$-connective). These localizations are related by natural transformations $L_n^f \ra L_{n-1}^f$, provided we arrange our choices so that the connectivity of 
$V_{n+1}$ is greater than or equal to the connectivity of $V_n$. Bousfield \cite{Bou01} chooses the $V_{n+1}$ so that their connectivity is as low as possible, but we will not make this restriction.

The map $X \to L_n^f X$ is a $v_i$-equivalence for $i \leq n$. Moreover, the $v_i$-periodic homotopy groups of $L_n^f X$ vanish for $i > n$, making the tower of $L_n^f$-localizations of $X$ analogous to a Postnikov tower.  The homotopy fiber $M_n^fX$ of the map $L_n^f X \ra L_{n-1}^f X$ then has the same $v_n$-periodic homotopy groups as $X$, but its $v_i$-periodic homotopy vanishes for $i \neq n$.  

\begin{defn}
The $\infty$-category $\mscr_n^f$ is the full subcategory of $L_n^f \mathcal{S}_\ast \ltri d_{n+1} \rtri$ on spaces that are of the form $M_n^f X \ltri d_{n+1} \rtri$.
\end{defn}

Bousfield's work \cite{Bou01} then yields the following characterization of $\mscr_n^f$ as a localization (see \cite[Theorem 2.2]{heuts-periodic}):

\begin{thm}
The $\infty$-category $\mscr_n^f$ is the localization of $\mathcal{S}_*$ at the $v_n$-periodic equivalences. More precisely, precomposition with the functor
\begin{equation*}
\mathcal{S}_* \ra \mscr_n^f: X \mapsto M_n^f X \ltri d_{n+1} \rtri
\end{equation*}
gives, for any $\infty$-category $\cscr$, an equivalence of $\infty$-categories
\begin{equation*}
\mathrm{Fun}(\mscr_n^f, \cscr) \rightarrow \mathrm{Fun}_{v_n}(\mathcal{S}_*, \cscr).
\end{equation*}
Here $\mathrm{Fun}_{v_n}$ denotes the full subcategory consisting of those functors sending $v_n$-equivalences to equivalences.
\end{thm}

It is important to note that while the embedding of $\mscr_n^f$ into $\mathcal{S}_*$ depends on the choice of $V_{n+1}$, the $\infty$-category $\mscr_n^f$ itself is well-defined up to equivalence. Indeed, this follows from the universal property of the preceding theorem. The localization $\mscr_n^f$ turns out to have good formal properties. In particular, \cite[Theorem 2.2]{heuts-periodic} guarantees that it is a presentable $\infty$-category, so that $\mscr_n^f$ has all colimits. Moreover, those colimits are preserved by the inclusion
\begin{equation*}
\mscr_n^f \rightarrow L_n^f\mathcal{S}_\ast \ltri d_{n+1} \rtri.
\end{equation*}

\begin{cor}
\label{cor:conservative}
The Bousfield-Kuhn functor factors through a functor
\begin{equation*}
\Phi: \mscr_n^f \rightarrow \Sp_{T(n)}
\end{equation*}
which is \emph{conservative}, i.e., a map $\varphi$ in $\mscr_n^f$ is an equivalence if and only if $\Phi(\varphi)$ is an equivalence.
\end{cor}
\begin{proof}
This is a consequence of the fact that $\varphi$ is a $v_n$-periodic equivalence if and only if $\Phi(\varphi)$ is an equivalence. To see this, one uses first that a map of $T(n)$-local spectra $E \rightarrow F$ is an equivalence if and only if $E^V \rightarrow F^V$ is an equivalence, with $V$ a finite type $n$ complex with $v_n$-self map $v\colon \Sigma^dV \to V$; this follows as $T(n)$ can be described as the mapping telescope of $v$. The result follows by the natural identification
\begin{equation*}
\pi_*(\Phi(X)^V) \cong v^{-1}\pi_*(X;V)
\end{equation*}
mentioned before.
\end{proof}

In fact, the spectrum $\Phi(X)^V$ can be described in a rather explicit way, which also makes the identification of its homotopy groups as in the preceding proof clear. One defines a spectrum $\Phi_v(X)$ by setting 
\begin{equation*}
\Phi_v(X)_0 = \mathrm{Map}_*(V,X), \, \Phi_v(X)_d = \mathrm{Map}_*(V,X), \,\ldots, \Phi_v(X)_{kd} = \mathrm{Map}_*(V,X), \,\ldots,
\end{equation*}
and using the maps
\begin{equation*}
\Phi_v(X)_{kd} = \mathrm{Map}_*(V,X) \xrightarrow{v^*} \mathrm{Map}_*(\Sigma^d V, X) \cong \Omega^d\Phi_v(X)_{(k+1)d}
\end{equation*}
as structure maps. Then $\Phi_v$ is the \emph{telescopic functor} associated to the self-map $v$. There is an equivalence of spectra (see \cite[Theorem 1.1]{KuhnTelescopic})
\begin{equation*}
\Phi(X)^V \simeq \Phi_v(X).
\end{equation*}


Bousfield shows (Theorem 5.4(i),(ii) of \cite{Bou01}) that on the level of homotopy categories the functor of the previous corollary admits a left adjoint $\Theta$. However, his techniques also prove the stronger result below. The necessary straightforward modifications are in the proof of Theorem 2.3 of \cite{heuts-periodic}. 

\begin{prop} The functor $\Phi$ admits a left adjoint
\begin{equation*}
\Theta: \Sp_{T(n)} \rightarrow \mscr_n^f.
\end{equation*}
\end{prop}


%
%

\section{Monads and the Barr-Beck-Lurie Theorem} 


For background on modules over monads in an $\infty$-categorical setting, the reader can consult \cite{RiehlVerity} or \cite{HA}. 

\begin{defn}\cite[Definition 4.7.0.1]{HA}
Let $\cscr$ be an $\infty$-category. A \textbf{monad} $M$ on $\cscr$ is an algebra object of $\Fun(\cscr, \cscr)$ with respect to the composition monoidal structure. If $M$ is a monad on $\cscr$, we let $\Alg_M(\cscr)$ denote the associated $\infty$-category of (left) $M$-modules in $\cscr$.
\end{defn}

\begin{defn}\cite[Definition 4.7.4.4.]{HA} Let $G :\dscr \ra \cscr$ be a functor between $\infty$-categories. Assume that $G$ has a left adjoint $F$, so that there is a corresponding monad $M \simeq G \circ F$ on $\cscr$. We will say that \textbf{$\dscr$ is monadic over
$\cscr$} if the induced functor $G : \dscr \ra \Alg_M (\cscr)$ is an equivalence of $\infty$-categories.
\end{defn}

Next we state Lurie's version of the Barr-Beck theorem, also known as the monadicity theorem. 
We only state a special case that we need here. 
\begin{thm}\cite[Theorem 4.7.0.3]{HA}
Suppose we are given a pair of adjoint functors
\[
\xymatrix{
\cscr \ar@<3pt>[r]^F & \dscr \ar@<3pt>[l]^G
}
\]
between $\infty$-categories where $\dscr$ admits geometric realizations of simplicial objects. Assume that
\begin{itemize} \itemsep -0.1cm
\item[(i)] $G$ is conservative and
\item[(ii)] $G$ preserves geometric realizations of simplicial objects. 
\end{itemize}
Then $\dscr$ is monadic over $\cscr$.
\end{thm}

\section{The Bousfield-Kuhn functor is monadic}


In this section we establish that the Bousfield-Kuhn functor $\Phi\colon \mscr_n^f \to \Sp_{T(n)}$ satisfies the hypotheses of the Barr-Beck-Lurie Monadicity Theorem stated in the previous section. Corollary \ref{cor:conservative} states that $\Phi$ is conservative. To apply the Barr-Beck-Lurie theorem, what remains is to establish the following:
\begin{prop} 
$\Phi: \mscr_n^f \ra \Sp_{T(n)}$ commutes with geometric realizations. 
\end{prop}

\begin{proof}
Consider a simplicial object $X_\bullet \in (\mscr_n^f)^{\Delta^{\mathrm{op}}}$. We should check that the map
\begin{equation*}
|\Phi(X_\bullet)| \ra \Phi(|X_\bullet|)
\end{equation*}
is an equivalence of $T(n)$-local spectra. Choose a finite type $n$ space $W$ with a $v_n$ self-map $w: \Sigma^d W \ra W$. Then it suffices to check that
\begin{equation*}
|\Phi(X_\bullet)|^W \ra \Phi(|X_\bullet|)^W \simeq \Phi_w(|X_\bullet|)
\end{equation*}
is an equivalence. On the left we may commute the exponent $W$ past geometric realization (because $W$ is finite), so that the left-hand side is equivalent to $|\Phi_w(X_\bullet)|$. In other words, it suffices to check that 
\begin{equation*}
\Phi_w: \mscr_n^f \rightarrow \Sp_{T(n)}
\end{equation*}
preserves geometric realizations.

Recall that the inclusion $\mscr_n^f \ra L_n^f\mathcal{S}_*\ltri d_{n+1}\rtri$ preserves colimits, so that the colimit of any diagram $\xscr$ in $\mscr_n^f$ may be computed as follows:
\begin{equation*}
\colim_{\mscr_n^f} \xscr \simeq \colim_{L_n^f\mathcal{S}_*\ltri d_{n+1}\rtri}\xscr \simeq L_n^f (\colim_{\mathcal{S}_*\ltri d_{n+1}\rtri}\xscr).
\end{equation*}
It follows that there is an equivalence
\begin{equation*}
\Phi_w(\colim_{\mscr_n^f} \xscr) \simeq \Phi_w(\colim_{\mathcal{S}_*\ltri d_{n+1}\rtri}\xscr),
\end{equation*}
since $Y \ra L_n^f Y$ is a $v_n$-equivalence for any space $Y$. Thus we have reduced to showing that
\begin{equation*}
\Phi_w: \mathcal{S}_*\ltri d_{n+1} \rtri \ra \Sp_{T(n)}
\end{equation*}
preserves geometric realizations.

The definition of $\Phi_w$ implies the formula
\begin{equation*}
\Phi_w X \simeq \colim ( \Sigma^\infty \Map_\ast (W,X) \ra \Omega^d \Sigma^\infty \Map_\ast (W,X) \ra \cdots ).
\end{equation*} 
Since $\Sigma^\infty$ and $\Omega^d$ preserve colimits, it is sufficient to check that the functor
\begin{equation*}
\Map_\ast(W,-): \mathcal{S}_*\ltri d_{n+1} \rtri \ra \mathcal{S}_*
\end{equation*}
preserves geometric realizations.

Recall that the localization $L_n^f$ involved the choice of a finite type $n+1$ suspension space $V_{n+1}$, \emph{which can be chosen freely} (as long as its connectivity is at least that of $V_n$). In particular, we may choose $V_{n+1}$ so that its connectivity is at least the dimension of $W$. We have now reduced to showing that $\Map_\ast(W,-)$ commutes with geometric realizations of diagrams of spaces all of which have connectivity at least the dimension of $W$. This follows from some rather classical homotopy theory, which we summarize in the proof of Proposition \ref{prop:realize} below. 


\end{proof} 

\begin{prop}\label{prop:realize}
Let $W$ be a finite CW complex and $X_\bullet \in \mathcal{S}_{\ast}^{\Delta^{op}}$ a simplicial space such that $X_n$ is $\dim(W)$-connective for all $n$. Then the natural map
\[
\chi_W\colon |\Map_\ast (W, X_\ast)| \ra \Map_\ast (W, |X_\ast|)
\]
is an equivalence.
\end{prop}

To prove Proposition \ref{prop:realize} we will make use of the following
lemma. Compare \cite[Theorem B.4]{BF} or \cite[Proposition 5.4]{Rezk-pistar}. 

\begin{lem}\label{lem:hopb-prod}
For a diagram of simplicial spaces
\[
\xymatrix{X_{\bullet} \ar[r]\ar[d] & E_{\bullet} \ar[d]\\ Y_{\bullet} \ar[r] & B_{\bullet}\\}
\]
which is a levelwise homotopy pullback, and where $B_n$ is connected for every $n$, the natural map
\[
|X_{\bullet}| \ra |E_{\bullet}| \x_{|B_{\bullet}|} |Y_{\bullet}|
\]
is an equivalence.
\end{lem}

The idea is to use the lemma and skeletal induction on $W$. 

\begin{proof}[Proof of Prop \ref{prop:realize}]:
Fix a simplicial space $X_{\bullet}$ such that each $X_n$ is $d$-connective.
First, recall that realization commutes with finite products, and 
$$\Map_\ast( A \Wedge B, X) \simeq \Map_\ast(A, X) \x \Map_\ast (B, X).$$
Therefore $\chi_{A\vee B}$ is an equivalence if $\chi_A$ and $\chi_B$ are equivalences. 

We consider the collection $\mathcal{C}$ of pointed spaces
$V$ such that the natural map $\chi_V$ is an equivalence. 
The previous paragraph implies that $\mathcal{C}$ is closed under finite
wedge sums. We claim that any finite CW complex $W$ of dimension $\leq d$ belongs
to $\mathcal{C}$. 

To see this, we use induction on 
$n = \dim W$. Clearly $S^0 \in \mathcal{C}$, so the fact that $\mathcal{C}$ is closed
under finite wedges covers the case $n = 0$. 
For $n>0$, we can write $W$ as a homotopy pushout \[
\xymatrix{
\bigvee S^{n-1} \ar[r]\ar[d] & \sk_{n-1} W\ar[d]\\
\ast \ar[r] & W.\\
}
\]
This leads to a homotopy cartesian diagram 
of simplicial spaces
\begin{equation*} 
\xymatrix{
\Map_\ast (W, X_{\bullet}) \ar[r]\ar[d] & \Map_\ast (\sk_{n-1} W, X_{\bullet}) \ar[d]\\
\ast \ar[r] & \Map_\ast (\BigWedge S^{n-1}, X_{\bullet}).
}
\end{equation*}
Our assumptions on $n$ and $d$ guarantee that the simplicial space $\Map_\ast (\BigWedge
S^{n-1}, X_{\bullet})$ is (levelwise) connected, so that the realization of the square above is still homotopy cartesian by Lemma \ref{lem:hopb-prod}. Consider the following cube, of which the left and right face are homotopy cartesian:
\[\scalebox{.85}{$
\xymatrix{
|\Map_\ast(W, X_{\bullet})| \ar[rr]^{\hspace{1.5cm}\chi_W}\ar[dd] \ar[dr]&& \Map_\ast (W, |X_{\bullet}|) \ar[dr]\ar[dd]\\
 & |\Map_\ast (\sk_{n-1} W, X_{\bullet})| \ar[rr]^{\hspace{-1.75cm}\chi_{\mathrm{sk}_{n-1}
W}}\ar[dd] && \Map_\ast (\sk_{n-1}W, |X_{\bullet}|)\ar[dd]\\
 \ast \ar[rr]^{\hspace{1.5cm}\chi_{\ast}}\ar[dr]&& \ast\ar[dr]\\
 & |\Map_\ast (\BigWedge S^{n-1}, X_{\bullet})| \ar[rr]^{\hspace{-1.75cm}\chi_{\bigvee S^{n-1}}}&& \Map_\ast (\BigWedge S^{n-1}, |X_{\bullet}|)
}$}
\] 
The horizontal maps $\chi_{\ast}, \chi_{\bigvee S^{n-1}}, $ and $\chi_{\mathrm{sk}_{n-1}
W}$ are equivalences by the inductive hypothesis, so that $\chi_W$ is an equivalence as well.
\end{proof}

\bibliographystyle{alpha}
\bibliography{biblio}

\end{document}